\documentclass[11pt]{amsart}
\usepackage{amsmath}
\usepackage{amssymb}
\usepackage{amsthm}
\usepackage{array}
\usepackage{xy}
\usepackage[pdftex]{graphicx}
\usepackage{hyperref}
\usepackage{color}
\usepackage{transparent}
\usepackage{latexsym}

\numberwithin{equation}{section}

\newtheorem{thm}{Theorem}[section]
\newtheorem{cor}[thm]{Corollary}
\newtheorem{lem}[thm]{Lemma}
\newtheorem{dfn}[thm]{Definition}
\newtheorem{prop}[thm]{Proposition}
\newtheorem{rem}[thm]{Remark}
\theoremstyle{definition}
\newtheorem{exmp}{Example}[section]

\begin{document}

\title[On the regularity of the free boundary in the optimal partial transport problem]{On the regularity of the free boundary in the optimal partial transport problem for general cost functions}

\author[S. Chen]{S.Chen}
\author[E. Indrei]{E. Indrei}


\date{}
\maketitle


\def\signei{\bigskip\begin{center} {\sc Emanuel Indrei\par\vspace{3mm}
MSRI\\  
17 Gauss Way\\
Berkeley, CA 94720\\
email:} {\tt eindrei@msri.org }
\end{center}}

\def\signsc{\bigskip\begin{center} {\sc Shibing Chen\par\vspace{3mm}
MSRI\\  
17 Gauss Way\\
Berkeley, CA 94720\\
email:} {\tt sbchen@math.utoronto.ca}
\end{center}}

\begin{abstract}
This paper concerns the regularity and geometry of the free boundary in the optimal partial transport problem for general cost functions. More specifically, we prove that a $C^1$ cost implies a locally Lipschitz free boundary. As an application, we address a problem discussed by Caffarelli and McCann \cite{CM} regarding cost functions satisfying the Ma-Trudinger-Wang condition (A3): if the non-negative source density is in some $L^p(\mathbb{R}^n)$ space for $p \in (\frac{n+1}{2},\infty]$ and the positive target density is bounded away from zero, then the free boundary is a semiconvex $C_{loc}^{1,\alpha}$ hypersurface. Furthermore, we show that a locally Lipschitz cost implies a rectifiable free boundary and initiate a corresponding regularity theory in the Riemannian setting.    
\end{abstract}

\section{Introduction}

In the optimal partial transport problem, one is given two non-negative functions $f=f\chi_\Omega, \hskip .1in g=g\chi_\Lambda \in L^1(\mathbb{R}^n)$ and a number $0<m \leq \min\{||f||_{L^1}, ||g||_{L^1}\}.$ The objective is to find an optimal transference plan between $f$ and $g$ with mass $m$. A transference plan refers to a non-negative, finite Borel measure $\gamma$ on $\mathbb{R}^n \times \mathbb{R}^n$, whose first and second marginals are controlled by $f$ and $g$ respectively: for any Borel set $A \subset \mathbb{R}^n$,
$$\gamma(A\times \mathbb{R}^n) \leq \int_{A} f(x)dx, \hskip .2in \gamma(\mathbb{R}^n \times A) \leq \int_{A} g(x)dx.$$
An optimal transference plan is a minimizer of the functional  
\begin{equation*} \label{mini}
\gamma \rightarrow \int_{\mathbb{R}^n \times \mathbb{R}^n} c(x,y) d\gamma(x,y),
\end{equation*}       
where $c$ is a non-negative cost function. 

Issues of existence, uniqueness, and regularity of optimal transference plans have recently been addressed by Caffarelli \& McCann \cite{CM}, Figalli \cite{AFi}, \cite{AFi2}, and Indrei \cite{I}. Indeed, existence follows readily by standard methods in the calculus of variations. However, in general, minimizers fail to be unique and it is not difficult to construct examples when $|spt (f \wedge g)|>0$ (with $|\cdot|$ being the Lebesgue measure and $spt(f \wedge g)$ the support of $f \wedge g:=\min \{f,g\}$). Nevertheless, Figalli proved that under suitable assumptions on the cost function, minimizers are unique for $$||f \wedge g||_{L^1(\mathbb{R}^n)}\leq m \leq \min\{||f||_{L^1(\mathbb{R}^n)}, ||g||_{L^1(\mathbb{R}^n)}\} \large,$$ \cite[Proposition 2.2 and Theorem 2.10]{AFi}.  Up to now, the regularity theory has only been developed for the quadratic cost. In this case, if the domains $\Omega$ and $\Lambda$ are bounded, strictly convex, and separated by a hyperplane, Caffarelli and McCann proved (under suitable conditions on the initial data) that the free boundaries $\overline{\partial U_m \cap \Omega}$ and $\overline{\partial V_m \cap \Lambda}$ are locally $C^{1,\alpha}$ hypersurfaces up to a closed singular set $\tilde S$ contained at the intersection of free with fixed boundary \cite[Corollary 7.15]{CM}; here, the free boundaries are generated by the sets $U_m$ and $V_m$ which are referred to as the ``active regions." $U_m$ is defined as the interior of the support of the left marginal of the optimal transference plan, and $V_m$ is similarly defined in terms of the right marginal (a characterization of these regions in terms of the cost function is given by \cite[Corollary 2.4]{CM}).  

In the case when there is overlap, Figalli proved that away from the common region $\Omega \cap \Lambda$, the free boundaries are locally $C^1$ \cite[Theorem 4.11]{AFi}; Indrei improved this result by obtaining local $C^{1,\alpha}$ regularity away from the common region and up to a relatively closed singular set $S$, necessarily contained at the intersection of fixed with free boundary, see \cite[Corollary 3.13]{I} for a precise statement. Moreover, under an additional $C^{1,1}$ regularity assumption on $\Omega$ and $\Lambda$, he proved that $S$ is $\mathcal{H}^{n-2}$  $\sigma$- finite and in the disjoint case $S \subset \tilde S$ with $\mathcal{H}^{n-2}(S) < \infty$ \cite[Theorem 4.9]{I}. 

All of the aforementioned regularity results were developed for the quadratic cost. Our main aim in this paper is to obtain free boundary regularity for a general class of cost functions $\mathcal{F}_0$ satisfying the Ma-Trudinger-Wang (A3) condition introduced in \cite{T1} and used in the development of a general regularity theory for the potential arising in the optimal transportation problem (see Definition \ref{class}). With this in mind, we establish the following theorem which readily implies $C_{loc}^{1,\alpha}$ regularity of the free boundary for the family $\mathcal{F}_0$ and thereby solves a problem discussed by Caffarelli and McCann \cite[pg. 676]{CM}:

\begin{thm} \label{lip} (Lipschitz regularity) 
Let $f=f\chi_\Omega$, $g=g\chi_\Lambda$ be non-negative integrable functions and $m \in \big(0,\min\{||f||_{L^1},||g||_{L^1}\}\big]$. Assume that $\Lambda$ is bounded and $c$-convex with respect to $\Omega$, where $c \in C^1(\mathbb{R}^n \times \mathbb{R}^n)$ and satisfies (\ref{l4}) and (\ref{eee1}). Then the free boundary in the optimal partial transport problem is locally a Lipschitz graph. 
\end{thm}

The proof of this theorem is based on a cone method: first, we utilize a result of Caffarelli and McCann \cite{CM} to prove that the active region is generated by level sets of the cost function. Thus, the free boundary is locally a suprema of these level sets (at least at the heuristic level). Then, thanks to the assumptions on $c$, we prove that the free boundary enjoys a uniform interior cone condition; this implies that it is locally a Lipschitz graph in some system of coordinates. To solve the problem discussed by Caffarelli and McCann, we connect the free normal with the solution of a generalized Monge-Amp\`{e}re equation for (A3) cost functions and employ regularity results established by Loeper \cite{Lo} and refined by Liu \cite{Li}.    

\begin{cor} \label{ath} ($C^{1,\alpha}$ regularity)
Let $f=f\chi_\Omega \in L^p(\mathbb{R}^n)$ be a non-negative function with $p \in (\frac{n+1}{2},\infty]$, and $g=g\chi_\Lambda$ a positive function bounded away from zero. Moreover, assume $c \in \mathcal{F}_0$, $m \in \big(0,\min\{||f||_{L^1},||g||_{L^1}\}\big]$, $\Omega$ and $\Lambda$ are bounded, $\Lambda$ is relatively $c$-convex with respect to $\Omega \cup \Lambda$, and $\overline{\Omega} \cap \overline{\Lambda} = \emptyset$. Then $\partial U_m \cap \Omega$ is locally a $C^{1,\alpha}$ graph, where $\partial U_m \cap \Omega$ is the free boundary arising in the optimal partial transport problem and $\alpha=\frac{2p-n-1}{2p(2n-1)-n+1}$.
\end{cor}

In fact, thanks to the method developed by Figalli \cite{AFi}, one can localize the problem and eliminate the disjointness assumption $\overline{\Omega} \cap \overline{\Lambda} = \emptyset$, see Corollary \ref{co1}. We note that to obtain the Lipschitz result, we only need the cost to be $C^1$; however, with merely a locally Lipschitz assumption, the free boundary can still be shown to be rectifiable, see Proposition \ref{rect}.   

The rest of the paper is organized as follows: in \S $2$, we state and prove some preliminary facts. Then in \S $3$, we proceed with the proof of Theorem \ref{lip} and in \S $4$ address the problem in a Riemannian setting.

\section{Preliminaries}
\begin{dfn} \label{cone}
Given an $(m-1)$-plane $\pi$ in $\mathbb{R}^m$, we denote a general cone with respect to $\pi$ by $$C_\alpha(\pi):= \{z \in \mathbb{R}^m: \alpha |P_\pi(z)| < P_{\pi^\perp}(z)\},$$ where $\pi \oplus \pi^\perp = \mathbb{R}^m$, $\alpha>0$, and $P_\pi(z)$ $\&$ $P_{\pi^\perp}(z)$ are the orthogonal projections of $z\in \mathbb{R}^m$ onto $\pi$ and $\pi^\perp$, respectively. 
\end{dfn}

\begin{dfn}
A domain $D$ is said to satisfy the uniform interior cone condition if there exists $\alpha>0$ and $\delta>0$ such that for all $x \in \partial D$ there exists $\nu_x \in \mathbb{S}^{n-1}$ so that $$(y+C_\alpha(\nu_x^\perp)) \cap B_\delta(x) \subset D \cap B_\delta(x),$$ for all $y \in \overline D \cap B_\delta(x)$. We define the profile of such domains to be the ordered pair $(\delta,\alpha)$.
\end{dfn}

\begin{dfn}
A domain $D \subset \mathbb{R}^n$ is said to satisfy a uniform interior ball condition if there exists $r>0$ such that for all $x \in \partial D$, there exists $\nu_x \in \mathbb{S}^{n-1}$ for which $B_r(x+r\nu_x) \subset D$.  
\end{dfn}

\begin{dfn} \label{class}
We denote by $\mathcal{F}$, the collection of cost functions $c: \mathbb{R}^n \times \mathbb{R}^n \rightarrow \mathbb{R}$ that satisfy the following three conditions: 
\vskip .1in 
\noindent 1. $c \in C^2(\mathbb{R}^n \times \mathbb{R}^n)$;\\
\vskip .05in 
\noindent 2. $c(x,y) \geq 0$ and $c(x,y)=0$ only for $x=y$;\\
\vskip .05in
\noindent 3. (A1) For $x,p \in \mathbb{R}^n$, there exists a unique $y=y(x,p)\in \mathbb{R}^n$ such that $\nabla_x c(x,y) = p$ (left twist); \indent  similarly, for any $y, q \in \mathbb{R}^n$, there exists a unique $x=x(y,q) \in \mathbb{R}^n$ such that $\nabla_yc(x,y)=q$ \indent (right twist).\\   
\vskip .05in
\noindent Furthermore, we denote by $\mathcal{F}_0$, the set of $C^4(\mathbb{R}^n \times \mathbb{R}^n)$ cost functions in $\mathcal{F}$ that satisfy:\\
\vskip .05in
\noindent 4. (A2) $\det(\nabla_{(x,y)} c) \neq 0$ for all $x,y \in \mathbb{R}^n$;\\
\vskip .05in
\noindent 5. (A3) For $x,p \in \mathbb{R}^n$, $$A_{ij,kl}(x,p) \xi_i\xi_j\eta_k\eta_l \geq c_0|\xi|^2|\eta|^2 \hskip .1in \forall \hskip.1in \xi, \eta \in \mathbb{R}^n, \langle \xi, \eta\rangle = 0, c_0>0,$$
where $A_{ij,kl}:= c^{r,k}c^{s,l}(c^{m,n}c_{ij,m}c_{n,rs}-c_{ij,rs})$, and $(c^{i,j})$ is the inverse matrix of $(c_{i,j})$.   
\end{dfn}

\begin{rem}
Some authors use the notation $(A3)_s$ in place of $(A3)$ in condition $5$ of Definition \ref{class}.  
\end{rem}

\begin{dfn}
A set $V \subset \mathbb{R}^n$ is $c$-convex with respect to another set $U \subset \mathbb{R}^n$ if the image $c_x(x,V)$ is convex for each $x\in U$.  
\end{dfn}

\begin{lem} \label{cone2}
Let $\Omega \subset \mathbb{R}^n$, $\Lambda \subset \mathbb{R}^n$ be two domains and $c \in C^1(\mathbb{R}^n \times \mathbb{R}^n)$; assume 
\begin{equation} \label{l4}
b_0:= \inf_{x \in \Omega, y \in \Lambda} c(x,y)>0. 
\end{equation}

\begin{equation} \label{eee1}
b_1:= \inf_{x \in \Omega, y \in \Lambda} |\nabla_x c(x,y)|>0,
\end{equation} 

Then for any $b \geq b_0$ and $y \in \overline{\Lambda}$, the domain $E_y^b:=\{x \in \Omega: c(x,y)<b\}$ satisfies a uniform interior cone condition with profile depending only on $b_0$, $||c||_{C^1}$, and the modulus of continuity of $c_x$. 
\end{lem}

\begin{proof}
Fix $y \in \overline{\Lambda}$ and consider $\phi(x):=c(x,y)$. Then for a fixed point $x_0 \in \{x\in \Omega:\phi(x)=b\}$, we choose a coordinate system such that $x_n$ is the direction of the normal to the level set pointing into the sublevel set $\{x \in \Omega:\phi(x)\leq b\}$ and $x_0$ is the origin. Let $0<\theta<\frac{\pi}{2}$ and note that if $x$ has angle $\theta$ with $e_n$, then
$$\phi(x)=\phi(0)+\nabla \phi(x) \cdot x + o(x) \leq \phi(0)-b_1|x|cos(\theta)+o(x).$$ Now since $c \in C^1(\mathbb{R}^n \times \mathbb{R}^n)$, by the uniform continuity of $c_x$ we have $o(x) \leq \frac{1}{2} b_1|x|cos(\theta)$, for $x \in B_\delta(0)$ and $\delta>0$ (depending on $b_1$, $\theta$, and the modulus of continuity of $c_x$). Thus, $\phi(x) < b$ when $x$ has angle at most  $\theta$ from $e_n$ and is in the $\delta$-ball around the origin.  
\end{proof}

\begin{rem} \label{dir}
By the positivity of $b_1$ in (\ref{eee1}), it follows that we may take $\nu_x:=-\frac{c_x(x,y)}{|c_x(x,y)|}$ as the direction of the cone at each point $x \in \partial E_y^b$ and $y \in \overline \Lambda$. 
\end{rem}

\begin{rem} \label{ya}
Note that since $c \in C^1$, for a sufficiently small $\delta>0$, we may take $\theta$ arbitrarily close to $\frac{\pi}{2}$ in the proof of Lemma \ref{cone2}. In other words, given any $\alpha>0$, there exists $\delta(\alpha)>0$ such that $(\delta(\alpha), \alpha)$ can be taken as a profile for the level sets of $c$ under the assumptions of Lemma \ref{cone2}.  
\end{rem}

\noindent If the domains $\Omega$ and $\Lambda$ have disjoint closures and $c(x,y)$ is a continuous cost function vanishing only on the diagonal $x=y$, then (\ref{l4}) follows readily. The next lemma gives sufficient conditions for (\ref{eee1}) to hold. 

\begin{lem} \label{conddd}
Let $\Omega \subset \mathbb{R}^n$, $\Lambda \subset \mathbb{R}^n$ be two domains and $c \in C^1(\mathbb{R}^n \times \mathbb{R}^n)$; assume 
$$
\inf_{x \in \Omega, y \in \Lambda} c(x,y)>0,$$
and suppose $c$ satisfies the left twist condition and condition $2$ in Definition $2.4$. Then
$$\inf_{x \in \Omega, y \in \Lambda} |\nabla_x c(x,y)|>0.$$
\end{lem}

\begin{proof}
Suppose on the contrary that there exists $(\bar x, \bar y) \in \overline{\Omega} \times \overline{\Lambda}$ for which $\nabla_x c(\bar x, \bar y)=0$. Let $\phi(x):= c(x, \bar x)$; using condition $2$, $\phi(x) \geq 0$ and $\phi(x)=0$ only for $x=\bar x$. Therefore, $\nabla_x c(\bar x, \bar x) = 0$, but by uniqueness, we must have $\bar x =\bar y$ (using the left twist condition), and this contradicts the positivity of $b_0$.
\end{proof}

\begin{lem} \label{l5}
Let $c \in \mathcal{F}$, and consider two domains $\Omega \subset \mathbb{R}^n$, $\Lambda \subset \mathbb{R}^n$ with disjoint closures; set 
$$
b_0 = \inf_{x \in \Omega, y \in \Lambda} c(x,y)>0.
$$
Then for any $b \geq b_0$ and $y \in \overline{\Lambda}$, the domain $E_y^b:=\{x \in \Omega: c(x,y)<b\}$ satisfies a uniform interior ball condition with radius $r=r(b_0, b_1, ||c||_{C^2})>0$, where $b_1$ is defined by (\ref{eee1}). 
\end{lem}

\begin{proof}
First, note that since $c \in \mathcal{F}$ we have $b_1>0$ by Lemma \ref{conddd}. Now for a fixed $y_0 \in \overline{\Lambda}$, denote $\phi(x):=c(x,y_0)$. Then for a fixed point $x_0 \in \{x \in \Omega:\phi(x)=b\}$, we choose a coordinate system such that $x_n$ is the direction of the normal to the level set pointing into the sublevel set $\{x \in \Omega:\phi(x)\leq b\}$ and $x_0$ is the origin. Now let $r:=\frac{b_1}{c_2}$, where $c_2=||c||_{C^2}$, and consider the ball $B_r$ centered at $(0,\ldots, r)$ with radius $r$. In particular $\partial B_r$ touches the origin. Now we will show that $B_r \subset \{x \in \Omega:\phi(x)<b\}$: indeed, it is simple to see that for $x \in B_r$, $\cos(\theta) > \frac{|x|}{2r}=\frac{|x|c_2}{2b_1}$, where $\theta$ is the angle between $x$ and $e_n$. Therefore, 
\begin{align*}
\phi(x) &\leq \phi(0)+ \nabla \phi(0) \cdot x + \frac{c_2}{2} |x|^2 \\
&= b - |\nabla \phi(0)| e_n \cdot x + \frac{c_2}{2} |x|^2\\
&< b - (b_1 |x|)\bigg(\frac{|x|c_2}{2b_1}\bigg) +\frac{c_2}{2} |x|^2 =b.\\ 
\end{align*}              
\end{proof}
\begin{rem}
By interchanging the roles of $x$ and $y$ in Lemma \ref{l5}, a similar statement holds for $E_x^b:=\{y \in \Lambda: c(x,y)<b\}$.          
\end{rem}

\section{Regularity theory}

\begin{proof}[\bf{Proof of Theorem \ref{lip}}]
By our assumptions, we have
\begin{equation} \label{yeah}
U_m \cap \Omega:= \bigcup_{(\bar x, \bar y) \in \gamma_m} \{x \in \Omega: c(x,\bar y) < c(\bar x, \bar y)\},
\end{equation}
see \cite[Corollary 2.4]{CM}. Next, let $x \in \partial U_m \cap \Omega$ and note that since $$x \in \partial \{x \in \Omega: c(x,\bar y) < c(\bar x, \bar y)\},$$ Lemma \ref{cone2} implies the existence of a profile $(\delta, \alpha)$ so that $$\big (x+C_\alpha(\nu_x^\perp) \big) \cap B_\delta(x) \subset (U_m \cap \Omega) \cap B_\delta(x),$$ where $\nu_x:=-\frac{c_x(x,T_m(x))}{|c_x(x,T_m(x))|}$ (see Remark \ref{dir}) and $T_m$ is the map in \cite[Lemma 2.3]{CM} on whose graph $\gamma_m$ is supported. Note that $T_m$ may be multi-valued (i.e. there may exist $(x,y_1)$ and $(x,y_2)$ in the support of $\gamma_m$ with $y_1\neq y_2$). In this case, we select one of them as the value of $T_m(x)$ (the key fact is: $\{T_m(x)\} \subset \overline{\Lambda}$). Now for $z \in \partial U_m \cap \Omega \cap B_\delta(x)$, consider the bounded, convex set $c_x(z,\Lambda)$ (the boundedness follows from the $C^1$ regularity of $c$ and the boundedness of $\Omega$ and the convexity is a result of the $c$-convexity assumption of $\Lambda$). As 
$$0<\inf_{z \in \Omega, y \in \Lambda} c_x(z,y):=b_1$$ (this follows from (\ref{eee1})), the origin is not in the closure of $c_x(x,\Lambda)$. Thus, we may find $\xi_x \in \mathbb{S}^{n-1}$ and $\omega_x>0$ so that $-c_x(x, \Lambda) \subset C_{\omega_x}(\xi_x^\perp)$ (i.e. we may trap a bounded, convex set whose closure does not containing the origin inside a cone of opening smaller than $\pi$). Up to possibly decreasing $\omega_x$ slightly, we may assume that the boundary of $-c_x(x, \Lambda)$ is disjoint from the boundary of $C_{\omega_x}(\xi_x^\perp)$; hence, by $C^1$ regularity of $c$, there exists $\delta_x>0$ so that $-c_x(z,\Lambda) \subset C_{\omega_x}(\xi_x^\perp)$ for all $z \in B_{\delta_x}(x)$; recall that for $z \in \partial U_m \cap B_{\delta_x}(x)$,
\begin{equation} \label{cnnq}
\big (z+C_{\alpha}(\nu_z^\perp) \big) \cap B_{\delta}(z) \subset (U_m \cap \Omega) \cap B_{\delta}(z),
\end{equation}
with $\nu_z \in C_{\omega_x}(\xi_x^\perp)$ $\big($since $\nu_z= -\frac{c_x(z, T_m(z))}{|c_x(z, T_m(z))|}$ and $-c_x(z,\Lambda) \subset C_{\omega_x}(\xi_x^\perp)$ $\big)$; now the angle between $\nu_z$ and $\nu$ is strictly less than $\frac{\pi}{2}$ due to the fact that the opening of $C_{\omega_x}(\xi_x^\perp)$ is strictly less than $\pi$. Thus, thanks to Remark \ref{ya}, we may assume $\xi_x \in C_{\alpha}(\nu_z^\perp)$ (by picking $\alpha>0$ sufficiently small). Hence, there exists $\alpha_x>\alpha$ so that (see Figure \ref{inter}) $$C_{\alpha_x}(\xi_x^\perp) \subset \bigcap_{z \in \partial U_m \cap B_{\delta_x}(x)} C_{\alpha}(\nu_z^\perp).$$ 
Combining this information with (\ref{cnnq}) yields that for all $z \in \partial U_m \cap B_{\delta_x}(x)$,
\begin{equation} \label{cnn}
\big (z+C_{\alpha_x}(\xi_x^\perp) \big) \cap B_{\delta_x}(x) \subset (U_m \cap \Omega) \cap B_{\delta_x}(x)
\end{equation} 
(we may assume $B_{\delta_x}(x) \subset B_\delta(z)$ by choosing $\delta_x$ sufficiently small relative to $\delta$). In fact, by possibly taking $\alpha_x$ larger and $\delta_x$ smaller, (\ref{cnn}) holds for all $z \in  \overline U_m \cap B_{\delta_x}(x)$: indeed, let $z \in  U_m \cap B_{\delta_x}(x).$ Then $(z,T_m(z)) \in \gamma_m$ and thanks to (\ref{yeah}), $z \in \partial E_{T_m(z)}^b$ with $b=c(z,T_m(z))>0$. Remark \ref{dir} implies the existence of $\nu_z=-\frac{c_x(z, T_m(z))}{|c_x(z, T_m(z))|} \in \mathbb{S}^{n-1}$ so that $$(y+C_\alpha(\nu_z^\perp)) \cap B_\delta(z) \subset E_{T_m(z)}^b \cap \Omega \subset U_m \cap \Omega,$$ for all $y \in  \overline{E_{T_m(z)}^b} \cap B_\delta(x)$. In particular, $$(z+C_\alpha(\nu_z^\perp)) \cap B_\delta(z)  \subset U_m \cap \Omega.$$ Thus, by possibly taking $\delta_x$ smaller, if necessary, we may assume $\overline B_{\delta_x}(x) \subset B_\delta(z)$, and if $z$ is close enough to $x$ we also have $\nu_z \in C_{\omega_x}(\xi_x^\perp)$; thus, repeating the argument above from (\ref{cnnq}) to (\ref{cnn}) yields the result. Therefore, we proved the existence of $\delta_x>0$, $\alpha_x>0$, and $\xi_x \in \mathbb{S}^{n-1}$ so that for all $z \in \overline U_m \cap \overline B_{\delta_x}(x)$, 
\begin{equation} \label{ra}
(z+C_{\alpha_x}(\xi_x^\perp))\cap B_{\delta_x}(x) \subset U_m \cap \Omega.
\end{equation}
\begin{figure}[h!]
\centering 
\includegraphics[scale= .5]{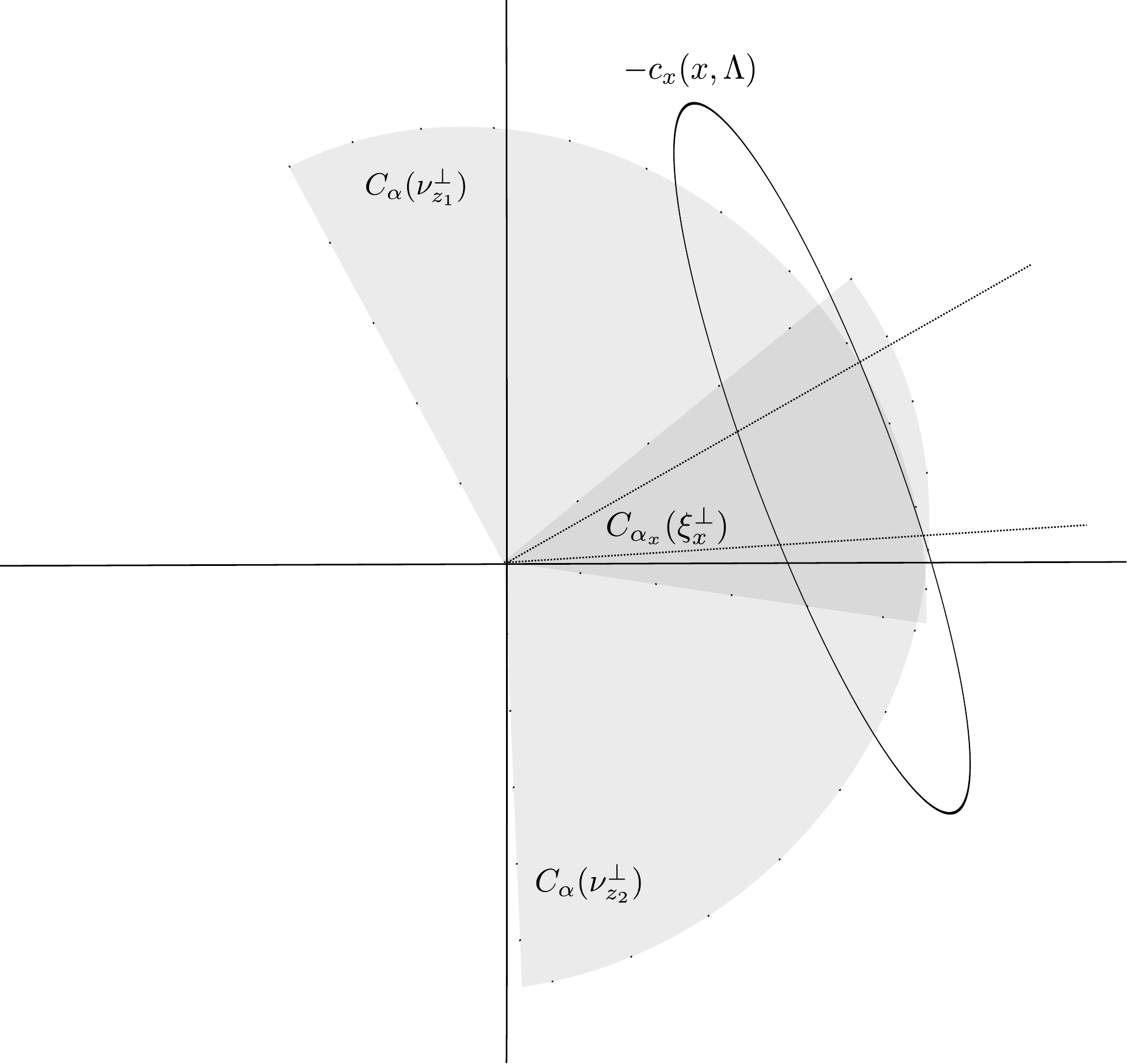}
\caption{$C_{\alpha_x}(\xi_x^\perp) \subset \bigcap_{z \in \partial U_m \cap B_{\delta_x}(x)} C_{\alpha}(\nu_z^\perp).$}
\label{inter}
\end{figure} 

\noindent By rotating and translating the coordinate system, we may assume $x=0$, $\xi_x = -e_n$, and $\pi := \xi_x^\perp = \mathbb{R}^{n-1}$; moreover, note that the cone $C_{\alpha_0}(\pi)$ is symmetric with respect to the $e_n$ axis. Define $\phi: \tilde B_{\eta_0}(0) \subset \mathbb{R}^{n-1} \rightarrow \mathbb{R}$ by $$\phi(z'):= \sup_{y:=(y',y_n) \in \partial U_m \cap \overline{B_{\eta_0}(0)}} K_y(z'),$$ where $\tilde B_{d\eta_0}(0):=Proj_\pi(B_{\eta_0}(0))$ and $K_y$ is the cone function at the point $y$ on the free boundary generated by $C_{\alpha_0}(\pi)$. Note that $\phi$ is Lipschitz since it is the supremum of Lipschitz functions with bounded Lipschitz constant (depending on the opening of the cones). 
Moreover, by construction we have 
\begin{equation} \label{cl0}
\partial U_m \cap B_{\eta_0}(0) \subset \operatorname{graph} \phi |_{\tilde B_{\eta_0}(0)}.
\end{equation}
Now we claim that there exist constants $d,\tilde d \in (0,1)$ with $d$ depending on the profile of the level sets of the cost function, so that 
\begin{equation} \label{cl}
\operatorname{graph} \phi |_{\tilde B_{d\eta_0}(0)} \subset \partial U_m \cap \overline{B_{\tilde d\eta_0}(0)}.
\end{equation}
Indeed, pick any $\tilde d \in (0,1)$; we may select a constant $d=d(\tilde d, \alpha_0)>0$ small enough, so that the graph of $\phi(\tilde B_{d\eta_0}(0))$ is contained in $B_{\tilde d \eta_0}(0)$ (this is possible, since $\phi$ has a uniform Lipschitz constant in $B_{\eta_0}(0)$ which depends only on the profile of the level sets).
Let $y \in \operatorname{graph} \phi |_{\tilde B_{d\eta_0}(0)}\subset B_{\tilde d \eta_0}(0)$. If $y \notin \partial U_m \cap \overline{B_{\tilde d\eta_0}(0)}$, then since $y$ is on an open cone with opening inward to $U_m \cap \Omega$, it follows that $y \in U_m \cap \Omega$. Since $\partial U_m \cap \overline B_{\tilde d\eta_0(0)}$ is compact,  for $\theta>0$ small, it follows that $Q_\theta(y) \cap \partial U_m \cap \overline B_{\tilde d\eta_0}  = \emptyset$, where $Q_\theta=Q_\theta(y)$ is a small cylinder whose interior is centered at $y$ and whose base diameter and height is equal to $\theta$; in particular, $Q_\theta \cap \operatorname{graph} \phi |_{\tilde B_{d\eta_0}(0)}$ does not contain any free boundary points. Next we consider a general fact: let $w \in \operatorname{graph} \phi |_{\tilde B_{\eta_0}(0)} \setminus \partial U_m$, $L_t(w) := w+te_n$, and $$s(w):= \sup_{\{t\geq0: L_t(w) \in U_m \cap \Omega\}} t;$$ 
note that since $w \in \operatorname{graph} \phi |_{\tilde B_{\eta_0}(0)}$,
\begin{equation} \label{le}
s(w)\geq \tilde s(w):=\sup_{\{t\geq0: L_t(w) \in B_{\eta_0}(0)\}} t,
\end{equation}
(otherwise it would contradict the definition of $\phi$ as a suprema of cones in $B_{\eta_0}(0)$ and $w$ as a point on the graph of $\phi$). 
Next, keeping the base fixed, we enlarge the height of the cylinder along the $\{y+te_n: t\in \mathbb{R}\}$ axis in a symmetric way (with respect to the plane $y_n+\pi = \mathbb{R}^{n-1}$) so that it surpasses $4\eta_0$; we denote the resulting cylinder by $\tilde Q_\theta$. By (\ref{le}) we have $\tilde Q_\theta \cap B_{\eta_0} \subset U_m \cap \Omega$. Then we increase its base diameter, $\theta$, until the first time when $\tilde Q_\theta$ hits the free boundary $\partial U_m \cap \Omega$ inside $B_{\eta_0}(0)$, and denote the time of first contact by $\theta$ and a resulting point of contact by $y_\theta$ (note that since $0 \in \partial U_m \cap B_{\eta_0}(0),$ this quantity is well defined). 
Since $\phi$ is a continuous graph in $B_{\eta_0}(0)$, and both $y$ and $y_\theta$ are on the graph, we may select a sequence of points $y_k \in \operatorname{graph} \phi |_{\tilde B_{\eta_0}(0)} \cap \tilde Q_\theta$ such that $y_k \rightarrow y_\theta$ (by connectedness of $\operatorname{graph} \phi |_{\tilde B_{\eta_0}(0)} \cap \tilde Q_\theta$). Since $y_\theta \in B_{\eta_0}(0)$ is an interior point, for $k$ sufficiently large we will have $y_k \in B_{\eta_0}(0) \cap \tilde Q_\theta$, see Figure \ref{cone_pic}. Thus, by definition of $\theta$, we will have that the $y_k$ are not free boundary points but on the graph of $\phi$; thus, by (\ref{le}), $s(y_k) \geq \tilde s(y_k)$, and this implies $\tilde y_k:= y_k+\tilde s(y_k)e_n \in \partial B_{\eta_0}(0) \cap \overline U_m$. By (\ref{ra}) we have $$(\tilde y_k+C_{\alpha_0}(\pi))\cap B_{\eta_0}(0) \subset U_m \cap \Omega.$$ However, for large $k$, $y_\theta \in (\tilde y_k+C_{\alpha_0}(\pi))$ (see Figure \ref{cone_pic}) and this contradicts that $y_\theta$ is a free boundary point, thereby establishing (\ref{cl}). Thus, combining (\ref{cl0}) and (\ref{cl}) we obtain that in a neighborhood around the origin, the free boundary is the graph of the Lipschitz function $\phi$; hence, the normal to the graph exists for $\mathcal{H}^{n-1}$ a.e. $z' \in \tilde B_{\eta_0}(0)$ and has the representation $\frac{(D\phi(z'), -1)}{\sqrt{1+|D\phi(z')|^2}}$ at a point $(z',\phi(z'))$ where it exists. 
\begin{figure}[h!]
\centering 
\includegraphics[scale= .5]{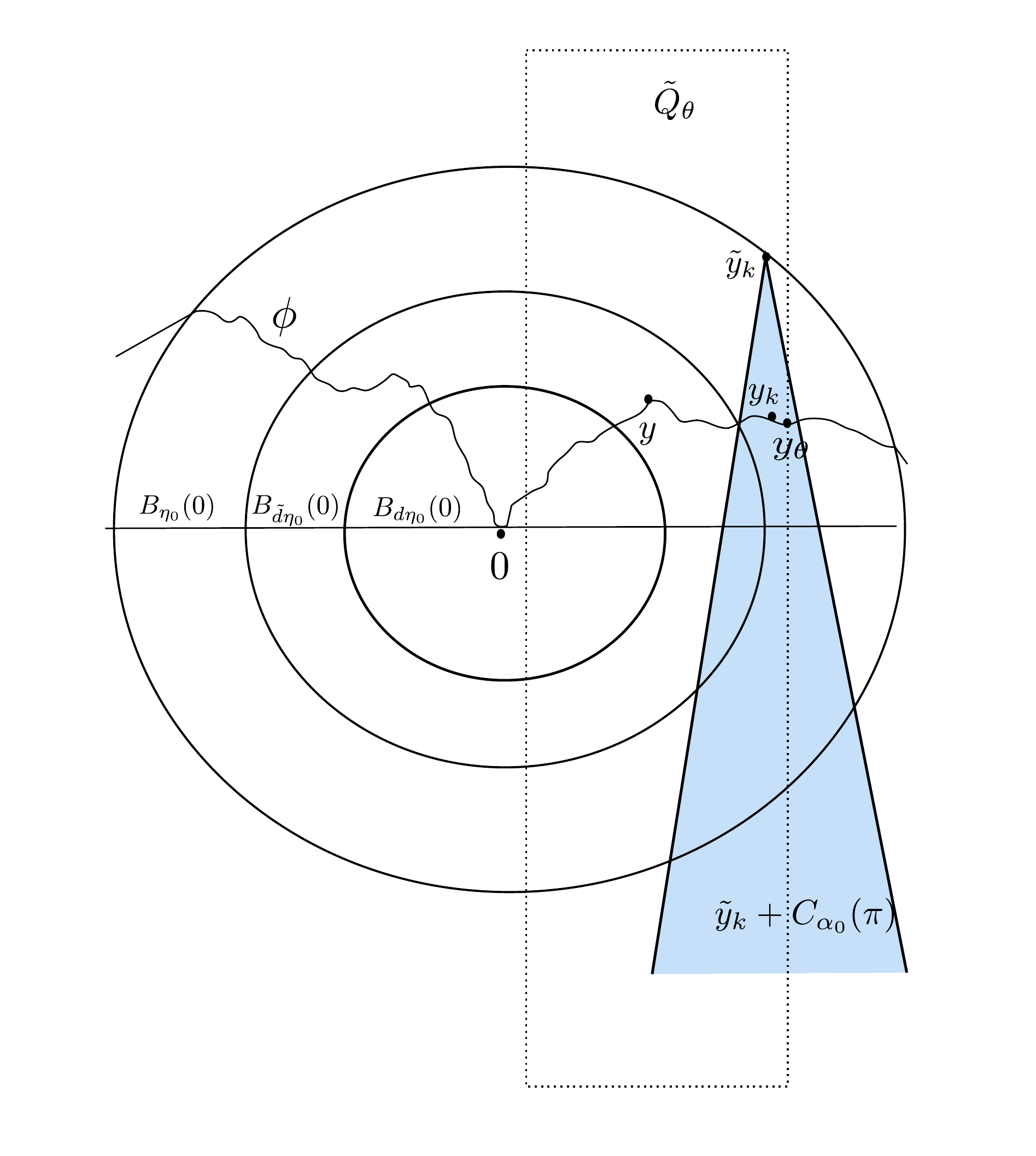}
\caption{$y_\theta \in U_m \cap \Omega$.}
\label{cone_pic}
\end{figure}   
  
\end{proof}

\noindent Now we are ready to apply our Lipschitz result to solve a problem mentioned in \cite{CM}:

\begin{proof}[\bf{Proof of Corollary \ref{ath}}]
First, note that by \cite[Remark 2.11]{AFi}, there exists a unique solution to the optimal partial transport problem, and it is denoted by the transference plan $\gamma_m$. Moreover, by the results of \cite[Section 2]{AFi}, this solution has the form $$\gamma_m:=(Id \times T_m)_{\#}f_m=(T_m^{-1} \times Id)_{\#} g_m,$$ for some invertible map $T_m$, where $f_m$ and $g_m$ are the first and second marginals of $\gamma_m$, respectively. Now $T_m$ is constructed by solving the classical optimal transport problem between the densities $f+(g-g_m)$, $g+(f-f_m)$. Indeed, if we denote this solution by $\gamma$, then since $c \in \mathcal{F}_0$, by applying the classical theory, we know $\gamma$ is supported on the graph of a function which is precisely $T_m$. Moreover, by \cite[Proposition 2.4]{AFi} and \cite[Remark 2.5]{AFi}, it follows that 
\begin{equation} \label{one}
\gamma = \gamma_m + (Id \times Id)_\#((f-f_m)+(g-g_m)),
\end{equation}
and there exists a potential function $\Psi_m$ which satisfies 
\begin{equation} \label{tran}
\nabla_x c(x,T_m(x))=\nabla \Psi_m(x),
\end{equation}
in an almost everywhere sense. Now by \cite[Theorem 2.6 and Remark 2.11]{AFi}, we have $(T_{m})_{\#}(f_m+(g-g_m))= g$ (i.e. $T_m$ will not move the points in the inactive region). Let $f':=f_m+(g-g_m)$ and note that thanks to our assumptions on $f$ and $g$, $$|det(D_{xy}^2 c)| \frac{f'}{g(T_m)} \in L^p(U_m \cap \Omega).$$ Thus, we may apply \cite[Theorem 1]{Li} to obtain 
\begin{equation} \label{hold}
\Psi_m \in C^{1,\alpha}(\overline{U_m \cap \Omega}),
\end{equation}
Now thanks to Theorem \ref{lip}, we know that the free boundary is locally a Lipschitz graph, with $\mathcal{H}^{n-1}$ a.e. defined normal $\nu_z= -\frac{c_x(z, T_m(z))}{|c_x(z, T_m(z))|}$. Thus by combining (\ref{tran}) and (\ref{hold}), we readily obtain the result.      

\end{proof}

In fact, one may also use Theorem \ref{lip} to prove a semiconvexity result. Moreover, the disjointness assumption may also be weakened.  

\begin{cor} \label{semi} (Semiconvexity) 
Let $f=f\chi_\Omega$ and $g=g\chi_\Lambda$ be a non-negative integrable functions. Assume $\overline{\Omega} \cap \overline{\Lambda}=\emptyset$ and that $\Lambda$ is bounded and $c$-convex with respect to $\Omega$. If $c \in \mathcal{F}$, then $\partial U_m \cap \Omega$ is locally semiconvex. 
\end{cor}

\begin{proof}
By Theorem \ref{lip}, it follows that $\partial U_m \cap \Omega$ is locally a Lipschitz graph, and Lemma \ref{l5} implies a uniform interior ball condition with $\nu_x= -\frac{c_x(x, T_m(x))}{|c_x(x, T_m(x))|}$ as the direction of the ball (c.f. Remark \ref{dir}). Thus, locally, the free boundary may be represented as a suprema of spherical caps (see \cite[Section 5]{CM}), and this readily implies semiconvexity.  
\end{proof}

\begin{cor} (Non-disjoint case) \label{co1}
Let $f=f\chi_\Omega \in L^p(\mathbb{R}^n)$ be a non-negative function with $p \in (\frac{n+1}{2},\infty]$, and $g=g\chi_\Lambda$ a positive function bounded away from zero. Moreover, assume that $\Omega$ and $\Lambda$ are bounded and $\Lambda$ is relatively $c$-convex with respect to a neighborhood of $\Omega \cup \Lambda.$ Let $c \in \mathcal{F}_0$ and $m \in \big(||f \wedge g||_{L^1}, \min\{||f||_{L^1},||g||_{L^1}\} \big]$. Then away from $\partial(\Omega \cap \Lambda)$, it follows that $\partial U_m \cap \Omega$ is locally a $C^{1,\alpha}$ graph, where $\partial U_m \cap \Omega$ is the free boundary arising in the optimal partial transport problem and $\alpha=\frac{2p-n-1}{2p(2n-1)-n+1}$. 
\end{cor}

\begin{proof}
Note that by \cite[Remark 3.2]{AFi} and \cite[Remark 3.3]{AFi}, we have $\Omega \cap \Lambda \subset U_m \cap \Omega$. Therefore, the free boundary does not enter the common region $\Omega \cap \Lambda$. Now let $x \in (\partial U_m \cap \Omega) \setminus \partial (\Omega \cap \Lambda)$. Choose $r_x>0$ so that $B_{r_x}(x) \cap (U_m \cap \Omega)$ does not intersect $\overline{\Lambda}$. Thus, $dist\big(B_{r_x}(x) \cap (U_m \cap \Omega),T_m(B_{r_x}(x) \cap (U_m \cap \Omega)\big)>0,$ and so we may apply Lemma \ref{l5} to obtain that all level sets of the cost function $c \in \mathcal{F}_0$ intersecting, say, $B_{\frac{r_x}{2}}(x) \cap (U_m \cap \Omega)$ have a uniform interior ball condition (in fact, a uniform interior cone condition is sufficient). Since all of Figalli's results used in the proof of Corollary \ref{ath} are also valid in the non-disjoint case, by localizing the problem in this way, we may proceed as in the proof of Corollary \ref{ath} to obtain the result.       
\end{proof}

\begin{rem}
By a localization argument, one may remove the disjointness assumption in Corollary \ref{semi}. Indeed, the precise statement (and proof) is similar to Corollary \ref{co1}.    
\end{rem}

\begin{rem} (Exchange symmetry) \label{w1}
By reverse symmetry, we may interchange the roles of $f$ and $g$ in Theorem \ref{ath} in order to obtain $C_{loc}^{1,\alpha}$ regularity of $\partial V_m \cap \Lambda$.    
\end{rem}

\begin{rem} (Geometry of $c$-convex domains)
For a geometric description of $c$-convex domains, see \cite[Section 2.1-2.3]{TW}. For example, based on a calculation therein, one can prove the following: suppose $\Lambda$ is a bounded, open convex set with smooth boundary and $\Omega \subset \mathbb{R}^n$. Let $$a_1:= \inf_{x \in \Omega \cup \Lambda, y \in \Lambda} |\det c_{x,y}(x,y)|>0$$ and $a_2:=||c(\cdot, \cdot)||_{C^3}$; for a fixed $x \in \Omega \cup \Lambda$, if the principal curvatures of $\partial \Lambda$ are greater than $\frac{a_2^n}{a_1}$, then $c_x(x,\Lambda)$ is convex.      
\end{rem}

Finally, we show that one may obtain a rectifiability result under only a locally Lipschitz assumption on the cost function.  

\begin{prop} \label{rect} (Rectifiability) 
Let $f=f\chi_\Omega,$ $g=g\chi_\Lambda$ be non-negative integrable functions and $m \in \big(0,\min\{||f||_{L^1},||g||_{L^1}\}\big]$. If $c:\mathbb{R}^n \times \mathbb{R}^n \rightarrow \mathbb{R}^+$ is locally Lipschitz in the $x$ variable, and 
\begin{equation} \label{alm}
0 \notin \partial_x c,
\end{equation}
where $\partial_x c$ is the Clarke subdifferential of $c$, then the free boundary arising in the optimal partial transport problem is $(n-1)$-rectifiable.
\end{prop}

\begin{proof}
First, by utilizing \cite[Corollary 2.4]{CM} we obtain    
$$
U_m \cap \Omega:= \bigcup_{(\bar x, \bar y) \in \gamma_m} \{x \in \Omega: c(x,\bar y) < c(\bar x, \bar y)\}.
$$ 
Next,  since $c$ is locally Lipschitz and (\ref{alm}) holds, we may apply the nonsmooth implicit function theorem \cite[Theorem 10.50]{Vi} to deduce that for $a\geq0$ and $\bar y \in \overline{\Lambda}$, the level set $$\{x \in \mathbb{R}^n: c(x,\bar y)=a\}$$ is locally an $(n-1)$-dimensional Lipschitz graph. Now, for $x \in \partial U_m \cap \Omega$, it follows that $$x \in \partial \{x \in \Omega: c(x,\bar y) < c(\bar x, \bar y)\},$$ for some $(\bar x, \bar y) \in \gamma_m.$ Hence, there exists a profile $(\delta_x, \alpha_x)$ such that 
\begin{equation} \label{prof}
\big (x+C_{\alpha_x}(\nu_x^\perp) \big) \cap B_{\delta_x}(x) \subset (U_m \cap \Omega) \cap B_{\delta_x}(x),
\end{equation}
for some $\nu_x \in \mathbb{S}^{n-1}.$ Consider the sets $$A_{j}^x:=\Big \{ z \in (\partial U_m \cap \Omega) \cap B_{\delta_x}(x): \delta_z \geq \frac{1}{j}, \alpha_z\leq j \Big\},$$ and note that by the argument leading to (\ref{prof}), each $z \in \partial U_m \cap \Omega$ has a profile $(\delta_z,\alpha_z)$. Now for each $j \in \mathbb{N}$, we may select $\epsilon_j>0$ so that $P:=\{\nu_i\}_{i=1}^{m_{\epsilon_j}}$ is a sufficiently fine partition of $\mathbb{S}^{n-1}$ in the following sense: for each $\nu \in \mathbb{S}^{n-1}$, there exists $\nu_i \in P$ so that $|\nu-\nu_i|< \epsilon_j$, where $\epsilon_j$ is chosen so that $|\nu-\nu_i|< \epsilon_j$ implies $\nu \in C_{2j}(\nu_i^\perp)$ (note: $\epsilon_j \searrow 0$ as $j \rightarrow \infty$). Next, let $$z \in A_{ij}^x:=\Big \{ z \in A_{j}^x: |\nu_z-\nu_i|<\epsilon_j \Big\}.$$ By using $\nu_z \in C_{2j}(\nu_i^\perp)$ and $\alpha_z\leq j$, it follows that there exists $\alpha_j>0$ so that $$C_{\alpha_j}(\nu_i^\perp) \subset C_j(\nu_z^\perp) \subset C_{\alpha_z}(\nu_z^\perp).$$ Moreover, since $\delta_z \geq \frac{1}{j}$, we may select $\delta_j>0$ and combine it with (\ref{prof}) (with $x$ replaced by $z$) to deduce $$(z+C_{\alpha_j}(\nu_i^\perp)) \cap B_{\delta_j}(x) \subset (U_m \cap \Omega) \cap B_{\delta_j}(x).$$ Thanks to this cone property, it is not difficult to show that for each $i,j \in \mathbb{N}$, $A_{ij}^x$ is contained on the graph of a Lipschitz function (generated by suprema of the cones with fixed opening given by $\alpha_j$). This shows that $A_{ij}^x$ is $(n-1)$-rectifiable. Moreover, $$(\partial U_m \cap \Omega) \cap B_{\delta_x}(x)=\bigcup_{j=1}^\infty \bigcup_{i=1}^{m_{\epsilon_j}} A_{ij}^x.$$  Next, let $(\partial U_m \cap \Omega)_s:=\{x \in \partial U_m \cap \Omega: dist(x, \partial \Omega)\geq s\}.$ By compactness, there exists $\{x_k\}_{k=1}^{n(s)} \subset (\partial U_m \cap \Omega)_s \subset \partial U_m \cap \Omega$ so that $$(\partial U_m \cap \Omega)_s = \bigcup_{k=1}^{n(s)} (\partial U_m \cap \Omega)_s \cap B_{\delta_{x_k}}(x_k).$$ From what we proved, it follows that $$(\partial U_m \cap \Omega)_s = \bigcup_{k=1}^{n(s)} \bigcup_{j=1}^\infty \bigcup_{i=1}^{m_{\epsilon_j}} A_{ij}^{x_k},$$ where each $A_{ij}^{x_k}$ is $(n-1)$-rectifiable. Thus, by taking $s \rightarrow 0$, we obtain the result.    
\end{proof}


\section{Extensions to Riemannian manifolds}

In this section, we study the partial transport problem on Riemannian manifolds where the cost is taken to be the square of the Euclidian distance $d$. Indeed, existence and uniqueness of the partial transport has been established by Figalli \cite[Remark 2.11]{AFi}. Therefore, our main concern here will be the regularity of the free boundary. In view of the method developed in the previous section, we will solely focus on giving sufficient conditions for local semiconvexity of the free boundary (for definitions, etc. regarding optimal transport in the Riemannian setting, the reader may e.g. consult \cite{Vi}):

\begin{thm} \label{rsemi} (Semiconvexity)
Let $M$ be a smooth $n$ -- dimensional Riemannian manifold with Riemannian distance $d$. Consider two non-negative, integrable functions $f=f\chi_\Omega$, $g=g\chi_\Lambda$ with $\overline{\Omega} \cap \overline{\Lambda}=\emptyset$ and $cut(\Omega) \cap \Lambda =\emptyset$, where $cut(\Omega)$ is the cut locus of $\Omega$. Furthermore, assume $\Lambda$ is bounded and $\frac{d^2}{2}$ - convex with respect to $\Omega$. Then the free boundary in the partial transport problem with cost $c:=\frac{d^2}{2}$ is locally semiconvex.   
\end{thm}   

\begin{proof}
First, by \cite[Remark 2.11]{AFi}, the partial transport exists and classical results imply that it has the form $T_m=\exp(\nabla \Psi_m)$ for some $c$ - convex function $\Psi_m$. Next, pick a free boundary point $x \in \Omega$. Then, for $\epsilon>0$ small, $\exp_x$ may be used as a chart between $B_\epsilon(x)$ and the tangent space at $x$. Since $cut(\Omega) \cap \Lambda =\emptyset$, 
$$cut(B_\epsilon(x) \cap (\partial U_m \cap \Omega)) \cap T_m(B_\epsilon(x) \cap (\partial U_m \cap \Omega))=\emptyset;$$ thus, we may use only one chart (i.e. $\exp$) to project $$\big(B_\epsilon(x) \cap (\partial U_m \cap \Omega)\big) \cup T_m(B_\epsilon(x) \cap (\partial U_m \cap \Omega))$$ onto the tangent space at $x$. The cut locus assumption also implies that $d$ is smooth on $\Omega \times \Lambda$, and since $\overline{\Omega} \cap \overline{\Lambda}=\emptyset$, we have that it is bounded away from zero. Thus, the level sets of $c$ enjoy a uniform ball condition (see Lemma \ref{l5}). Moreover, thanks to the $c$ - convexity of $\Lambda$ with respect to $\Omega$, we may apply Theorem \ref{lip} to deduce local Lipschitz regularity of the free boundary (note that (\ref{eee1}) follows from Lemma \ref{conddd} which we can apply thanks to our cut locus assumption). The Lipschitz regularity combines with the uniform ball property of the level sets and readily yields local semiconvexity of the free boundary.   
\end{proof}

\begin{rem}
We note that one may remove the disjointness assumption in Theorem \ref{rsemi} by localizing the problem as in Corollary \ref{co1}. 
\end{rem}

If $cut(\Omega)\cap \Lambda \neq \emptyset$, then it may happen that $T_m(x) \in cut(x)$, so the proof above breaks down (since the distance function is only smooth away from the cut locus). Nevertheless, there is currently some literature available in understanding the proper conditions which ensure that this scenario does not happen. Indeed, Loeper and Villani \cite[Theorem 7.1]{LV} shed some light on this issue: given a uniformly regular manifold (see \cite[Definition 4.1]{LV}) and two densities $\mu$ and $\nu$ such that $\mu << dvol$ and $\nu(A) \geq a vol(A)$ for any Borel set $A \subset M$, there is a $\sigma>0$ depending on $\mu$ and $\nu$, so that 
$$ \inf_{x\in M} d(T(x),cut(x)) \geq \sigma.$$ It is well-known that $\mathbb{S}^{n-1}$ is a uniformly regular manifold (see \cite[Example 4.3]{LV}), so one may hope to utilize Loeper and Villani's theory to develop a full regularity theory for the partial transport first on the sphere, and then in a more general setting.    
We conclude with a family of examples which illustrate that neither the ``stay away from the cut locus" property nor the $c$ - convexity of the target with respect to the entire source are necessary conditions for semiconvexity of the free boundary. 

\begin{exmp}
In what follows, we outline a method for constructing two general densities $f=f\chi_\Omega$, $g=g\chi_\Lambda$ where $\Omega \subset \mathbb{S}^{n-1}$ and $\Lambda \subset \mathbb{S}^{n-1}$ so that $\Lambda$ is not $\frac{d^2}{2}$ - convex with respect to $\Omega$, yet the free boundary in the partial transport problem is locally semiconvex away from the common region: let $\Omega=\mathbb{S}^{n}$ and
$\Lambda$ be a small spherical cap centered around the south pole with height, say,
$\frac{1}{16}$ (measured from the south pole). It is not difficult to see that $c_x(N, \Lambda)$ is an annulus (here, $N$ is the north pole); hence, $\Lambda$ fails to be $c:=\frac{d^2}{2}$ - convex with respect to $\Omega$. Assume
$\int_{\Lambda}f$ is slightly smaller than $\int_{\Lambda}g$ (to ensure the existence of a free boundary), and the mass
$m$ transported is slightly larger than
$\int_{\Lambda}f$.
Now, enlarge $\Lambda$ to a bigger spherical cap  $\widetilde{\Omega}$
with height $\frac{1}{8}$ so that
$\int_{\widetilde{\Omega}}f>\int_{\Lambda}g+\epsilon$, where $\epsilon$ is
a small positive constant (this can be accomplished by adjusting $f$ and $g$ at the
beginning). Then, it can be shown that the spherical cap
$\Omega_{1}$ centered at the north pole with height $\frac{1}{16}$ (measured from the north pole) is
outside the active region:
if not, suppose $\int_{\Omega_{1}\cap U_{m}}f_{m}>0$, and choose a subset
$A$ of $\Omega_{1}\cap U_{m}$ so that $\int_{A}f_{m}=\delta<\epsilon$;
then choose a subset $B$ of $\widetilde{\Omega}\setminus \Lambda$ so that
$\int_{B}f=\delta$. In the original partial transport plan, $A$ is
transported to $T_{m}(A),$ and it is easy to see that the distance between
$A$ and $T_{m}(A)$ is bigger than $2-\frac{1}{8}=\frac{15}{8}$. Now we modify the original plan by replacing the mass in $A$ by the mass in $B$,
and from $B$ to $T_{m}(A)$, we can cook up a new transport map thanks to 
the mass balance condition. Next, let $2\theta$ be the largest distance on $\widetilde{\Omega}$, and note $\theta\leq \tan(\theta)=\frac{\sqrt{1-(\frac{7}{8})^{2}}}{\frac{7}{8}}=\frac{\sqrt{15}}{7}\leq
\frac{4}{7}.$ Thus, $2\theta\leq \frac{8}{7} < \frac{15}{8}\leq d(A, T_m(A))$. Therefore, it is not difficult to see the new plan is cheaper than
the original, contradicting optimality. Thus, the original partial transport problem
is equivalent to a new one with source
$\Omega \setminus \Omega_{1}$ and target $\Lambda$; in the new problem we do
not have a cut locus issue -- this ensures an interior ball condition; moreover, it is not difficult to see that if $x$ is a free boundary point away from the common region, $c_x(x, \Lambda)$ is contained in a cone on the tangent space with vertex at $x$ whose opening is strictly less than $\pi$. Thus, we may proceed as in the proof of Theorem \ref{rsemi} to obtain local semiconvexity of the free boundary away from the common region. 
\end{exmp}


\noindent \textbf{Acknowledgments.} This work was initiated while the authors were participating in the 2012 AMSI/ANU/UQ Winter School on Geometric Partial Differential Equations in Brisbane, Australia. The excellent research environment provided by the winter school is kindly acknowledged. E. Indrei was supported by an NSF EAPSI fellowship during various stages of this work and would like to thank his EAPSI host Neil Trudinger for suggesting this line of research. The authors are also grateful to Robert McCann and Alessio Figalli for intellectually stimulating discussions and for their careful remarks on a preliminary version of the paper. 

\vskip 1in

\signsc
\signei
\end{document}